\documentclass[9 pt]{amsart}
\usepackage[utf8]{inputenc}
\usepackage {amsmath}
\usepackage {amsthm}
\usepackage {enumerate}
\usepackage{amssymb}
\usepackage{amsfonts}
\usepackage[all]{xy}
\usepackage{graphicx}
\usepackage{anysize}
\usepackage{setspace}

\newcommand{\N}{\mathbb N}
\newcommand{\Q}{\mathbb Q}

\numberwithin{equation}{section}

\theoremstyle{plain}

\newtheorem{teo}[equation]{Theorem}
\newtheorem{prop}[equation]{Proposition}
\newtheorem{cor}[equation]{Corollary}
\newtheorem{lema}[equation]{Lemma}
\newtheorem{pre}[equation]{Question}

\newtheorem{defi}[equation]{Definition}

\newtheorem{obs}[equation]{Remark}

\theoremstyle{definition}

\theoremstyle{definition}

\theoremstyle{definition}

\theoremstyle{definition}

\theoremstyle{definition}

\begin{document}
\title[
Symmetric products of the  $\mathfrak{E}_{\mathrm{c}}$ and the $\mathfrak{E}$.]{\scshape\bfseries 
Symmetric products of
 Erd\H{o}s space and complete Erd\H{o}s space.
}

\author{{ \bfseries }Alfredo Zaragoza }

\address{Departamento de Matematicas\\
         Facultad de Ciencias\\
         Universidad Nacional Aut\'onoma de M\'exico\\
        }
         \email{soad151192@icloud.com}



\date{}

\thanks{This work is part of the doctoral work of the author at UNAM, Mexico city, under the direction of R. Hernández-Gutiérrez and A. Tamariz-Mascarúa.
}

\begin{abstract}

It is shown that the symmetric products of complete Erd\H{o}s space and  Erd\H{o}s space are homeomorphic to complete Erd\H{o}s space and Erd\H{o}s space,
respectively. We will also give some properties of their hyperspace of compact subsets with the Vietoris topology.
\end{abstract}

\maketitle

\section{Introduction}

Every topological space in this article is assumed to be metrizable and separable.
The following two spaces were introduced by Erd\H{o}s in 1940 in \cite{er} as examples of totally disconnected and non-zero dimensional spaces. The Erd\H{o}s space is defined as follows
$$ \mathfrak{E} = \{(x_n)_{n\in \omega} \in  \ell^2 : x_i \in  \Q , \textit{for all } i\in \omega \};$$
and the complete Erd\H{o}s space as 
$$ \mathfrak{E}_{\mathrm{c}} = \{(x_n)_{n\in \omega} \in  \ell^2 : x_i \in  \{0\}\cup \{1/n: n\in \N\} \textit{ for all } i\in \omega \}$$
when $\ell^2$ is the Hilbert space of all square summable real sequences. The name \textit{complete Erd\H{o}s} space was introduced by  Kawamura, Oversteegen, and Tymchatyn \cite{K}. 

Some properties of the complete Erd\H{o}s space and Erd\H{o}s space were studied in 
\cite{j}, \cite{jj}, \cite{J3},  \cite{T} and \cite{K}. For a space $X$, $\mathcal{K}(X)$ denotes the hyperspace of non-empty compact subsets of $X$ with the
Vietoris topology; for any $n\in  \N$, $\mathcal{F}_n(X)$ 
is the subspace of $\mathcal{K}(X)$ consisting
of all the non-empty subsets that have cardinality less or equal to $n$; and $\mathcal{F}(X)$ is the subspace of $\mathcal{K}(X)$ of finite subsets of $X$. 
For $n\in \N$ and subsets $U_1,\ldots, U_n$ of a topological space $X$, we denote by $\langle  U_{1},\ldots ,U_{n}\rangle$ the collection $\left\lbrace   F \in \mathcal{K}(X):F\subset \bigcup_{k=1}^n U_k, F\cap U_{k}\neq \emptyset \textit{ for } k \leq n \right\rbrace $. Recall that the Vietoris topology on $\mathcal{K}(X)$ has as its canonical base all the sets of the form $\langle  U_{1},\ldots ,U_{n}\rangle$ where $U_k$ is a non-empty open subset of $X$ for each $k\leq n$.

We will study the properties that $\mathfrak{E}_c$ and $\mathfrak{E}$ have in common with some of their hyperspaces with the Vietoris topology. 
The main results of this work are:

\begin{teo}\label{ee1}
For any $n\in \N$, $\mathcal{F}_n(\mathfrak{E}_{\mathrm{c}})$ 
is homeomorphic to $\mathfrak{E}_{\mathrm{c}}$.

\end{teo}

\begin{teo}\label{EE1}
For any $n\in \N$, $\mathcal{F}_n(\mathfrak{E})$ 
is homeomorphic to $\mathfrak{E}$.

\end{teo}

Note that if $X$ is a separable metrizable space, then they every subspace of $\mathcal{K}(X)$ is also a separable metrizable space (see \cite[Theorem 3.3 and Propositions 4.4 and 4.5.2]{Sub})

\section{ Cohesive and almost zero dimension spaces}
An important property the of spaces $\mathfrak{E}$ and $\mathfrak{E}_{\mathrm{c}}$ is that at each point there is a local base such that each element of the base is an intersection of clopen sets. A space with this property is called
\textit{almost zero-dimensional} $(AZD)$. The fallowing result is well known. 

\begin{prop}[{\cite[Remark 2.4]{jj}}]\label{eazd} 
A topological space $(X,\mathcal{T})$ is almost zero-dimensional if
there is a zero-dimensional topology $\mathcal{W}$ in $X$ such that $\mathcal{W}$ is coarser than $\mathcal{T}$ and has the property
that every point in $X$ has a local neighborhood base consisting of sets that are closed with respect
to $\mathcal{W}$.

\end{prop}

If $(X,\mathcal{T})$ is an $AZD$ space, we say that the topology $\mathcal{W}$ which appears in Proposition \ref{eazd} witnesses the almost zero-dimensionality of $X$. Every zero-dimensional space is an $AZD$ space and the property of being $AZD$ is hereditary. On the other hand it is proved in \cite{Lh} that all AZD spaces have dimension less than or equal to one. Moreover, it is proved in \cite{jj} that all $AZD$ spaces are embeddable both in $\mathfrak{E}_c$ and $\mathfrak{E}$.

\begin{prop}\label{HZD}
For a topological space $X$ the following statements are equivalent.
\begin{enumerate}
\item $X$ is an $AZD$ space.
\item $\mathcal{K}(X)$ is an $AZD$ space.
\item If $\mathcal{F}_1(X)\subset\mathcal{A}\subset \mathcal{K}(X)$, then $\mathcal{A}$ is an $AZD$ space.

\end{enumerate}
\end{prop}
\begin{proof}

The implication $(2)\Rightarrow (3)$ is obvious, and $(3)\Rightarrow (1)$ follows from the fact that $\mathcal{F}_1(X)$ is homeomorphic to $X$.
\\ 
$(1)\Rightarrow (2) :$ We are going to prove that $\mathcal{K}(X)$ 
satisfies the conditions of Proposition \ref{eazd}.
Let $\mathcal{W}$ be a topology which witnesses the almost zero-dimensionality of $X$. 
Consider the space $Y=(X, \mathcal{W})$. As $\mathcal{W}$ is coarser than the topology of $X$, then $\mathcal{K}(X)\subset \mathcal{K}(Y)$. Let $(Z,\mathcal{W}_0)$ be the space $\mathcal{K}(X)$ considered with the topology inherited as a subspace of $\mathcal{K}(Y)$. Since
$\langle V_1,\ldots, V_n\rangle \cap Z$ is an open subset of $\mathcal{K}(X)$, when $V_1, \ldots, V_n$ are elements in $\mathcal{W}$, 
we have that the topology $\mathcal{W}_0$ of $Z$ is coarser than the topology of the $\mathcal{K}(X)$.
Moreover, by Proposition 4.13.1 in \cite{Sub}, $(Z,\mathcal{W}_0)$ is zero-dimensional.
Now, we are going to prove that each element in $\mathcal{K}(X)$ has a local neighborhood  base consisting of closed subsets in $(Z,\mathcal{W}_0)$. Let $F\in\mathcal{K}(X)$ and let $\mathcal{U}= \langle U_1,\ldots, U_m\rangle$ be
a canonical open subset of $\mathcal{K}(X)$ such that $F\in \mathcal{U}$. 
 For each $x\in F$ there is a neighborhood $V_x$ of $x$ in $X$ such that $x \in V_x\subset \bigcap\{U_j: x\in U_j\}$ and $V_x$ is closed in $Y$. Then $\{int_X(V_x):x\in F\}$ is an open cover of $F$ in $X$. As $F$ is a compact subset of $X$, there exist $x_1,\ldots, x_k\in F$ such that $F\subset \bigcup_{i=1}^k V_{x_i}$. For each $i\leq m$, let $y_i\in F\cap U_i$. Note that $F \in \langle V_{x_1},\ldots, V_{x_k}, V_{y_1}, \ldots, V_{y_m}\rangle$. Let us see that $\mathcal{V}_{\mathcal{U}}:=\langle V_{x_1},\ldots, V_{x_k}, V_{y_1}, \ldots, V_{y_m}\rangle\cap \mathcal{K}(X)\subset \mathcal{U}$. Indeed, let $H\in \mathcal{V}_{\mathcal{U}}$. Then $H\subset \bigcup_{i=1}^k V_{x_i} \cup \bigcup_{j=1}^m V_{y_j}$ and $H\cap V_{z}\neq\emptyset$ for $z\in \{x_1,\ldots, x_k, y_1,\ldots y_m\}$. 
By the choise of $V_{x_i}$ and $V_{y_j}$, we have $  \bigcup_{i=1}^k V_{x_i} \cup \bigcup_{j=1}^m V_{y_j}\subset \bigcup_{i\leq m} U_{i} $, and for each $j\leq m$ there exists a $z\in \{ x_1,\ldots, x_k, y_1,\ldots y_m\}$ such that $V_z\subset U_j$. Then $H\subset \bigcup_{i\leq m} U_{i}$ and $H\cap U_i\neq \emptyset$ for $i\leq m$. Thus $\mathcal{V}_{\mathcal{U}}\subset \mathcal{U}$, moreover $\langle V_{x_1},\ldots V_{x_k}\rangle$ is a closed subset in $\mathcal{K}(Y)$ (see \cite[Lemma 2.3.2 , p. 156]{Sub}), so $\langle V_{x_1},\ldots V_{x_k}\rangle \cap Z$ is a closed subset of $Z$.
Therefore, the collection of all the sets $\mathcal{V}_{\mathcal{U}}$ where $\mathcal{U}$ is a canonical open set of $\mathcal{K}(X)$ containing $F$, form a neighborhood local base of $F$ consisting of closed sets in $(Z, \mathcal{W}_0)$. Hence, by Proposition \ref{eazd}, $\mathcal{K}(X)$ is an $AZD$ space.

\end{proof}
The following Corollary follows from Proposition 4.13.1 in \cite{Sub} and Proposition \ref{HZD}, above. 
\begin{cor}\label{azc}
Let $X$ be an $AZD$ space, then $dim(\mathcal{K}(X))\leq 1$ and $dim(\mathcal{K}(X))= 1$ if and only if $dim(X)=1$.
\end{cor}
The omission of the hypothesis of almost zero-dimensionality on $X$ in the previous corollary produces the following natural question.

\begin{pre}
Is there a space $ X $ that is not $ AZD $ such that $ dim (X) = dim(\mathcal{K}(X))= 1$?
\end{pre}

In \cite{er} Erd\H{o}s proved that every point in $ \mathfrak{E} $ has a neighborhood that does not contain (nonempty) clopen sets. Later in \cite{jj} J.J Dijkstra and J. van Mill formalized this concept as follows.

\begin{defi}[{\cite[Definition 5.1]{jj}}] 
Let $X$ be a space and let $\mathcal{A}$ be a collection of subsets of $X$.
The space $X$ is called $\mathcal{A}$-cohesive if every point of the space has a neighborhood
that does not contain nonempty proper clopen subsets of any element of $\mathcal{A}$. If a space $X$
is $\{X\}$-cohesive then we simply call $X$ cohesive.
\end{defi}

Note that all connected spaces are cohesive and that the dimension of a cohesive space is greater than or equal to one. 
If $Y$ is any space and $X$ is a space that is $ \mathcal{A} $ - cohesive then $X \times Y$ is  $\{A \times B: A \in \mathcal{A}$ and $ B \subset Y \}$-cohesive (see {\cite[Remark 5.2, p. 21]{jj}}). Particularly if $ X $ is $ \{A_s:s\in S\} $ - cohesive, then $ X^n $ is $ \{A_{s_1}\times \ldots \times A_{s_n}: s_j \in S\} $-cohesive for any $ n \in \N $.

The following result relates the cohesion property between a space $ X $ and its symmetric products.

\begin{prop}\label{A2}
Let $n\in \N$ and let $X$ be a space $\{A_s: s\in S\}$-cohesive and $f:X^n\to \mathcal{F}_n(X)$ the function given by $f(x_1,\ldots,x_n)=\{x_1,\ldots x_n\}$. Then $\mathcal{F}_n(X)$ is $\{f[A_{s_1}\times \ldots \times A_{s_{n}}]: s_1,\ldots, s_n\in S\}$-cohesive. 

\end{prop}
\begin{proof}
Suppose that $\mathcal{F}_n(X)$ is not $\{f[A_{s_1}\times \ldots \times A_{s_{n}}]: s_1,\ldots, s_n\in S\}$-cohesive. Then there exist $F\in \mathcal{F}_n(X)$ and a local base $\beta$ of $F$, such that any $\mathcal{U}\in \beta$ 
contains a non-empty proper clopen
subset of some element of $\{f[A_{s_1}\times \ldots \times A_{s_{n}}]: s_1,\ldots, s_n\in S\}$.
Let us suppose that $F=\{x_1, \ldots ,x_k\}$ with $x_j\neq x_i$ for each $i,j\in \{1,\ldots k\}$. Let $\beta_{x_1}, \ldots, \beta_{x_k}$ be local bases of the points $x_1, \ldots, x_k$ respectively such that if $i\neq j$, $U_j\in \beta_{x_j}$ and $U_i\in \beta_{x_i}$, then $U_j\cap U_i=\emptyset$. 
Let $x=(x_1, \ldots,x_{k}, x_{k+1},\ldots, x_n)$ be in $X^n$, where $x_k=x_{k+1}=\ldots =x_n$.
Note that $\beta_0=\{U_1 \times\ldots\times U_{k+1}\ldots \times U_n: U_j\in \beta_{x_j}  \textit{ } j\leq k \textit{ y } U_k=U_{k+1}=\ldots= U_n\}$ is a local base at $x$. Let $V_j\in \beta_{x_j}$ be fixed for each $j\in\{i,\ldots k\}$.
Note that $F\in \mathcal{N}=\langle V_1, \ldots, V_k\rangle$ and that $x\in V_1\times \ldots \times V_n $. By our assumption there are $s_{1},\ldots, s_{n}\in S$ an open subset $\mathcal{U}\in \beta$, and a non-empty proper clopen subset $\mathcal{V}$ of $f[A_{s_{1}}\times \ldots \times A_{s_{n}}]$ such that $\mathcal{V}\subset \mathcal{U} \subset \mathcal{N}$.
 
Let $g=f\downharpoonleft A_{s_{1}}\times \ldots \times A_{s_{{n}}}$. As $g$ is continuous, we have that $g^{\leftarrow}[\mathcal{V}]$ 
is a clopen subset of $A_{s_{1}}\times \ldots \times A_{s_{n}}$. Let $C=g^{\leftarrow}[\mathcal{V}]\cap (V_1\times \ldots \times V_n)$. We are going to show that $C$ is a nonempty proper clopen subset of $A_{s_1}\times \ldots \times A_{s_{n}}$. 

Let $ E \in  \mathcal{F}_n(X)$, if $ E\in \mathcal{V}$, then there exist $w \in g^\leftarrow[E] $ such that $w \in (V_1\times \ldots \times V_n)$. If $E \in \mathcal{N}\setminus \mathcal{V}$, then
$g^\leftarrow[E]  \cap  g^{\leftarrow}[\mathcal{V}]= \emptyset$. Then there exist $w \in g^\leftarrow [E]$ such that $w \in A_{s_1}\times \ldots \times  A_{s_n} \setminus C$. It is clear that $C$ is a open subset of $A_{s_{1}}\times \ldots \times A_{s_{n}}$. 
To see that is closed let us consider a sequence $\{(y_1^m,\ldots, y_n^m): m\in \N\}$ of points of $ C $ such that $(y_1^m,\ldots, y_n^m)\to (y_1,\ldots, y_n)$. Since $f$ is continuous, the sequence $\{f((y_1^m,\ldots, y_n^m))=g((y_1^m,\ldots, y_n^m)):m\in \omega\}$ converges to $f(y_1,\ldots, y_n)$. Note that for every $m\in \omega$, we have that $f((y_1^m,\ldots, y_n^m))\in \mathcal{V}$. Since $\mathcal{V}$ is closed in $f[A_{s_{1}}\times \ldots \times A_{s_{n}}]$, we have that $f(y_1\ldots ,y_n)\in \mathcal{V}$.
Hence $(y_1,\ldots, y_n)\in f^{\leftarrow}[\mathcal{V}]$.
 On the other hand, as $\mathcal{V}\subset N$, then $y_j\in V_j$ as $y_j^m\in V_j$ for $m\in \N$. Thus $C$ is clopen in $A_{s_{1}}\times \ldots \times A_{s_{n}}$. This is a contradiction, so $X^n$ is $\{A_{s_1}\times\ldots\times A_{s_n}: s_{1}, \ldots ,s_{n}\in S\}$-cohesive.

\end{proof}

The next concept is a property that implies the cohesion of a space $ X $ and also relates the properties of cohesion and almost zero-dimensionality.

\begin{defi}
A one-point connectification of a space $X$ is a connected extension $Y$ of the
space such that the remainder $Y \setminus X$ is a singleton.

\end{defi}

The following result gives necessary and sufficient conditions for a metric separable space $X$ to have a metric and separable one-point connectification.
\begin{teo} (Knaster \cite{ub})\label{KKK}
Let $X$ be a separable metric space. Then $X$ has a one-point connectification $Y$ which is metrizable and separable if only if $X$ is embeddable 
in a separable metric connected space $Z$ 
as an proper open subset of $Z$.
\end{teo}

Recall that if $U$ is an proper open subset of $X$, then $\mathcal{K}(U)$, $\mathcal{F}_n(U)$ and $\mathcal{F}(U)$ are proper 
open subsets to $\mathcal{K}(X)$, $\mathcal{F}_n(X)$ and $\mathcal{F}(X)$ 
respectively. With this fact in mind, we can prove the following result.

\begin{prop} \label{a5}

If $X$ has a metrizable and separable one-point connectification, then $\mathcal{K}(X)$, $\mathcal{F}_n(X)$ and $\mathcal{F}(X)$ have, each, metric and separable one-point connectification. 
\end{prop}
\begin{proof}
If $Y$ is a metric and separable one-point connectification of $X$, then $\mathcal{K}(Y)$, $\mathcal{F}_n(Y)$ and $\mathcal{F}(Y)$ are metric and separable connected spaces (see \cite[Theorem 4.10]{Sub}). As $X$ is an proper open subset of $Y$, then $\mathcal{K}(X)$, $\mathcal{F}_n(X)$ and $\mathcal{F}(X)$ are 
proper open subsets to $\mathcal{K}(Y)$, $\mathcal{F}_n(Y)$ and $\mathcal{F}(Y)$ respectively. 
By Theorem \ref{KKK} we have that $\mathcal{K}(X)$, $\mathcal{F}_n(X)$ and $\mathcal{F}(X)$ have, each, a metric and separable one-point connectifications.

\end{proof}

It is known that if a space admits a one-point connectification, then it is cohesive. Moreover if an almost zero-dimensional space is cohesive, then it admits a one point connectification (see \cite[Proposition 5.4, p. 22]{jj})

\begin{cor} \label{a6}
If $X$ is a cohesive $AZD$ space, then  $\mathcal{F}_n(X), \mathcal{F}(X), \mathcal{K}(X)$ are cohesive $AZD$ spaces.
\end{cor}
\begin{proof}
As $X$ is cohesive $AZD$ space, then $X$ has a one-point connectification. As $X$ is a cohesive $AZD$ space, it has a
one-point connectification. By Proposition \ref{a5} $\mathcal{F}_n(X), \mathcal{F}(X), \mathcal{K}(X)$ have a one-point connectification. Thus $\mathcal{F}_n(X), \mathcal{F}(X), \mathcal{K}(X)$ are cohesive. Furthermore, by the Proposition, \ref{HZD} $\mathcal{F}_n(X), \mathcal{F}(X), $ and $ \mathcal{K}(X)$ are $AZD$.
\end{proof}

\begin{cor}
$\mathcal{F}_n(\mathfrak{E}_{\mathrm{c}})$, $\mathcal{F}_n(\mathfrak{E})$,
$\mathcal{F}(\mathfrak{E}_{\mathrm{c}})$, $\mathcal{F}(\mathfrak{E})$,
 $\mathcal{K}(\mathfrak{E}_{\mathrm{c}})$ and $\mathcal{K}(\mathfrak{E})$ 
are cohesive $AZD$ spaces.
 \end{cor}

\begin{proof}

This result follows from Corollary \ref{a6} and from the fact that $\mathfrak{E}_{\mathrm{c}}$ and $\mathfrak{E}$ are cohesive $AZD$ spaces.
\end{proof}

To prove that $\mathcal{K}(X)$ is cohesive in the Corollary \ref{a6}, it was assumed that $X$ is an $AZD$ space. Then, the following natural question.

\begin{pre}
Is there a $ X $ space cohesive such that $ \mathcal{K}(X) $ is not cohesive?
\end{pre}

\section{Proof of Theorem 1.1 }

\begin{defi}
Let $X$ be a set and let $ \tau_1$ and $ \tau_2$ 
two topologies in $X$, if $(X,\tau_1, \tau_2)$ 
satisfies that:
\begin{enumerate}
\item  $\tau_1 \subset \tau_2$ and $\tau_1$ 
is a zero dimensional topology such that every point in $ X$ has a neighborhood in $\tau_2$ which is compact with respect to $\tau_1$, we say that $(X,\tau_1, \tau_2)$ 
has property $C_1$.
\item  $\tau_1 \subset \tau_2$ and $\tau_1$ 
is a zero dimensional topology such that every point in $ X$ has a neighborhood in $\tau_2$ 
which is complete with respect to $\tau_1$, we say that $(X,\tau_1, \tau_2)$ 
has property $C_2$.
\end{enumerate}
\end{defi}

Note that the property $ C_1 $ is inherited to closed subsets with respect to $ \tau_1 $ and the property $ C_2 $ is inherited to the $ G_\delta $ subsets with respect to $ \tau_1 $. 

To prove Theorem 1.1 we will use the following characterization of $\mathfrak{E}_{\mathrm{c}}$.
\begin{teo}[{\cite[Theorem 3.1, items (1) and (2) ]{j}}]\label{EC} Let $ (\mathcal{E} ,\tau) $ be a topological space. 
The following statements are equivalent.
\begin{enumerate}
\item $\mathcal{E}$ 
is homeomorphic to $\mathfrak{E}_{\mathrm{c}}$.
\item  $ \mathcal{E} $ is cohesive and there exist a topology $\mathcal{W}$ in $ \mathcal{E}$ such that $(\mathcal{E}, \tau ,\mathcal{W})$ has property $C_1$.

\end{enumerate}
  
\end{teo}

For a topological space $(X, \tau)$,
the symbol $\tau_{\mathcal{F}_n(X)}$ denotes the Vietoris topology in $\mathcal{F}_n(X)$.

\begin{obs}\label{masg}
Let $X$ be a set and let $\tau_1$ and $\tau_2$  be
two topologies in $ X $ such that $\tau_1 \subset \tau_2$, then we have to $\tau_{1{\mathcal{F}_n(X)}}\subset \tau_{2{\mathcal{F}_n(X)}}$ 
\end{obs}
\begin{proof}

Let $\langle U_1,\ldots, U_n\rangle \in \tau_{1{\mathcal{F}_n(X)}}$, then $U_j\in \tau_1$ for each $j\in \{1,\ldots n\}$. As $\tau_1\subset \tau_2$, then $U_j\in \tau_2$ for each $j\in \{1,\ldots n\}$. This implies that $\langle U_1,\ldots, U_n\rangle \in \tau_{2{\mathcal{F}_n(X)}}$. So $\tau_{1{\mathcal{F}_n(X)}} \subset \tau_{2{\mathcal{F}_n(X)}}$.
\end{proof}

With this remark it is not difficult to see that if $\tau$ witnesses the almost zero-dimensionality
of $X$, then $\tau_{\mathcal{F}_n(X)}$ 
witnesses the almost zero-dimensionality of $\mathcal{F}_n(X)$.

In the proof of the following result, we use Propositions 2.4.2 and 4.1.3 and Theorem 4.2 in \cite{Sub}.

\begin{prop} \label{Pc1}

Let $X$ be a set and let $ \tau_1$ and $ \tau_2$ be two topologies on $X$. If $(X,\tau_1, \tau_2)$ has the property $C_1$ or $C_2$, 
then for any $n\in \N$, we have that $(\mathcal{F}_n(X), \tau_{1\mathcal{F}_n(X)}, \tau_{2\mathcal{F}_n(X)})$ has the property $C_1$ or $C_2$, 
respectively.
\end{prop}

\begin{proof}
First of all, since $\tau_1$ is a zero-dimensional topology on $X$, then $\tau_{1\mathcal{F}_n(X)}$ is a zero dimensional topology on $\mathcal{F}_n(X)$. Now, let $F\in \mathcal{F}_n(X)$ and $\mathcal{U}\in \tau_{2{\mathcal{F}_n(X)}}$ such that $F\in \mathcal{U}$. 
Let us suppose that $\mathcal{U}= \langle U_1,\ldots, U_m\rangle$ and that $F=\{x_1,\ldots, x_l\}$. For each $x_k\in F$ there exist $V_k\in \tau_2$ 
and a compact (resp., complete) $K_k$ of $X$ with respect to $\tau_1$ such that $x_k\in V_k \subset K_k\subset \bigcap\{U_j:x_k\in U_j\}$. Then $F\in \langle V_{1},\ldots ,V_{l}\rangle\subset \langle K_{1},\ldots ,K_{l}\rangle \subset \mathcal{U}$.
Note that $Z=\bigcup_{j=1}^l K_{j}$ 
is compact (resp., complete) with respect to $\tau_1$. Then $\mathcal{F}_n(Z)$ is compact(r esp., complete) with respect to $\tau_{1{\mathcal{F}_n(X)}}$. Moreover $\langle K_{1},\ldots ,K_{l}\rangle \subset \mathcal{F}_n(Z)$. Given that $\langle K_{1},\ldots K_{m}\rangle$ 
is a closed subset with respect to $\tau_{1{\mathcal{F}_n(X)}}$, we have that $\langle K_{1},\ldots K_{m}\rangle$ is compact   (resp., complete) with respect that $\tau_{1{\mathcal{F}_n(X)}}$. All this proves that $(\mathcal{F}_n(X), \tau_{1\mathcal{F}_n(X)}, \tau_{2\mathcal{F}_n(X)})$ has the property $C_1$ (resp.,$C_2$).

\end{proof}

\textbf{
Proof of  Theorem 1.1}
\begin{proof}
Let $\mathcal{W}$ be a topology in $\mathfrak{E}_{\mathrm{c}}$ which satisfies the conditions in the item (2) of Theorem \ref{EC}. By Remark \ref{masg}, $\mathcal{W}_{\mathcal{F}_n(\mathfrak{E}_{\mathrm{c}})}$ is coaser than the Vietoris topology on $\mathcal{F}_n(\mathfrak{E}_{\mathrm{c}})$. By Proposition \ref{Pc1} and the Corollary \ref{a6}, $\mathfrak{E}_{\mathrm{c}} $ satisfies all conditions in item (2) of Theorem \ref{EC}. Thus, by Theorem \ref{EC} $\mathcal{F}_n(\mathfrak{E}_{\mathrm{c}})$ is homeomorphic to $\mathfrak{E}_{\mathrm{c}}$.

\end{proof}

\section{Proof of Theorem 1.2}


To state the characterization of $\mathfrak{E}$ (see Theorem \ref{EE}) that we are going to use to prove Theorem 1.2 we will give some definitions.

\begin{defi} [{\cite[Definition 3.3, p. 8]{jj}}] If $A$ is a nonempty set then $A^{<\omega}$ denotes the set of all finite strings of elements of $A$, including the null
string $\emptyset$. If $s \in A^{<\omega}$ then $\vert s\vert$ denotes its length. In this context the set $A$ is called alphabet. Let $A^\omega$ denote the set of all
infinite strings of elements of $A$.
If $s \in A^{<\omega}$ and $t \in A^{<\omega} \cup A^\omega$,
then we put $s \prec t$ if $s$ is an initial substring of $t$; that is, there is an $r \in A^{<\omega} \cup A^\omega $ 
with $ s^\frown r = t$, where $\frown$ denotes concatenation of strings.
If $t\in A^{<\omega} \cup A^\omega$ and $k \in \omega$, 
$t\downharpoonleft k \in A^{<\omega}$ 
 is the element of $A^{<\omega} $ characterized by 
 $ t \downharpoonleft k \prec t$ and $\vert t\downharpoonleft k\vert =k$.
\end{defi}

\begin{defi}[{\cite[Definition 3.4, p. 9]{jj}}]\label{tree}A tree $T$ on $A$ is a subset of $ A^{<\omega}$ that is closed under initial segments, i.e., if $s \in T$ and $t \prec s$ then
$t \in T$. Elements of $T$ are called nodes. An infinite branch of $T$ is an element $r$ of $A^\omega$ such that
 $r\downharpoonleft k\in T$ for every $k \in \omega$. 
 The body of $T$, written as $[T]$, is the set of all infinite
  branches of $T$. If $s, t \in T$ are such that $s \prec t$
  and $\vert t\vert = \vert s\vert + 1$, 
 then we say that $t$ is an immediate successor of $s$ and $succ(s)$ denotes the set of immediate successors of $s$ in $T $. 
 \end{defi}

\begin{defi} 
 Let $n\geq 2$. If $S_1,\ldots, S_n$ are trees over $A_1,\ldots, A_n$, respectively, and if $s_1=a_1^1,\ldots, a_k^1\in S_1,\ldots, s_n=a_1^n,\ldots, a_k^n\in S_n$, are strings of equal length, then we difine the string $s_1*\ldots*s_n$ over $A_1\times \ldots \times A_n$ by $s_1*\ldots*s_n=(a_1^1,\ldots,a_1^n),\ldots,(a_k^1,\ldots,a_k^n)$. We define the product tree $S_1*\ldots *S_n$ over $A_1\times \ldots\times A_n$ as the partially ordered subset $\{s_1*\ldots*s_k:s_i\in S_i\textit{ for all i}\in \{1,\ldots, n\} \textit{ and } \vert s_1\vert \ldots= \vert s_n\vert \}$ of the alphabet $(A_1\times \ldots\times A_n)^{<\omega}$.
\end{defi}

\begin{obs}\label{tree1}
Let $n\geq 2$, and let $S_1,\ldots , S_n$ trees over $A_1,\ldots , A_n$, respectively. Then $S_1*\ldots * S_n$ is indeed a tree over $A_1\times \ldots \times A_n$. Moreover, the function $\phi: S_1*\ldots* S_n\to S_1\times \ldots \times S_n$ defined by $\phi(s_1*\ldots*s_n)=(s_1,\ldots, s_n)$, is an order isomorphism from $S_1*\ldots* S_n$ to $\phi [S_1*\times * S_n]$.... So, the following statements hold.
\begin{enumerate}
\item Let $s_1, t_1\in S_1, \ldots s_n, t_n\in S_n$ with $\vert s_1\vert= \ldots =\vert s_n\vert$ and $\vert t_1\vert= \ldots =\vert t_n\vert$. Then $s_1*\ldots *s_n\prec_{S_1*\ldots* S_n} t_1*\ldots* t_n$ if only if $s_i\prec_{S_i}t_i$ for all $i\in \{1,\ldots, n\}$.
\item Let $s_1, t_1\in S_1, \ldots s_n, t_n\in S_n$ with $\vert s_1\vert= \ldots =\vert s_n\vert$ and $\vert t_1\vert= \ldots =\vert t_n\vert$. Then $t_1*\ldots t_n\in succ(s_1*\ldots s_n)$ in $S_1*\ldots *S_n$ if only if $t_i\in succ(s_i)$ in $S_i$ for all $i\in \{1,\ldots, n\}$.

\item The body of $S_1*\ldots *S_n$, $[S_1*\ldots *S_n]$, is equal to the set $\{(\widehat{s_1},\ldots,\widehat{s_n} ):\widehat{s_k}\in [S_k] \textit{ for all } i\in \{1,\ldots, n\}\}$.
\item Let $\widehat{t_1}\in [S_1],\ldots, \widehat{t_n}\in [S_n]$, and let $k\in \omega$. Then 
$$(\widehat{t_1},\ldots,\widehat{t_n})\downharpoonleft k= (\widehat{t_1}\downharpoonleft k,\ldots,\widehat{t_n}\downharpoonleft k)=\widehat{t_1} \downharpoonleft k *\ldots*\widehat{t_n} \downharpoonleft k.$$
\end{enumerate}
\end{obs}

\begin{defi}
Let $X$ be a space. Let $(A_n)_{n\in \omega}$ 
a sequence of sets of $X$. We say that $(A_n)_{n\in \omega}$ converges to $x$ 
if for each open subset $U$ such that $x\in U$ there exists $m\in \omega$ such that $A_k\subset U$ if $m\leq k$.
\end{defi}

\begin{lema}\label{ancla}

Let $ f: X \to Y $ a continuous function and let $(A_n)_{n\in \omega}$ be a sequence of sets of $ X $ converging to $ x $ in $ X $, then $(f[A_n])_{n\in \omega}$ 
converges to $f(x)$.
\end{lema}

\begin{lema}\label{dnp}

Let $ f: X \to Y $ be a continuous and surjective function. 
Let $ A $ and $ B $ be subsets of $ Y $, such that $f^{\leftarrow}[A]$  is nowhere dense in $f^{\leftarrow}[B]$, then $A$  is nowhere dense in $B$
\end{lema}

\begin{proof}

Let us suppose $A$ is not nowhere dense in $B$. Then there exists $y\in A$ and an open subset $V$ of $Y$ such that $y\in B\cap V \subset A$. Let $x\in X$ be such that $f(x)=y$, then $x\in f^{\leftarrow}[B]\cap f^{\leftarrow}[V]\subset f^{\leftarrow}[A] $, which contradicts the hypothesis.
\end{proof}

\begin{defi}[{\cite[Definition 8.1]{jj}}]
Let $T$ be a tree and let $(X_s)_s\in T $ be a system of subsets of a space $X$ (called a scheme) such that $X_t\subset X_s$ whenever $ s\prec t$. 
A subset $A$ of $X$ is called an anchor for $(X_s)_s\in T $ in $X $ if either  for every $t\in [T ]$ we have $X_{t\downharpoonleft k} \cap A = \emptyset$ for some
$k \in \omega$ or the sequence $X_{t\downharpoonleft k_0},\ldots, X_{t\downharpoonleft n}\ldots $ converges to a point in $X$.
\end{defi}

\begin{teo}[{\cite[Theorem 8.13 , p. 46]{jj}}]\label{EE} A nonempty space $E$ is homeomorphic to $\mathfrak{E}$ if and only if there exists topology $\mathcal{W}$ on $E$ that witnesses the almost zero-dimensionality of $E$ and there exist a nonempty tree $T$ over a countable alphabet and subspaces $E_s$ of $E$ that are closed with
respect to $\mathcal{W}$ for each $s\in T$ such that:

\begin{enumerate}
\item $E_{\emptyset} = E$ and $E_s =\bigcup \{E_t : t \in succ(s)\}$ whenever $s \in T$,
\item each $x\in E$ has a neighborhood $U$ that is an anchor for $(E_s)_{s\in T}$ in $(E, \mathcal{W})$

\item for each $s \in T$ and $t \in succ(s)$, we have that $E_t$ is nowhere dense in $E_s$,
and
\item  $E$ is $\{Es : s \in T\}$-cohesive.
\end{enumerate}
Then $E$ is homeomorphic to $\mathfrak{E}$

\end{teo}

By the Theorem \ref{EE} 
there is a topology $\mathcal{W}$ for $\mathfrak{E}$ which is witnesses to the almost zero dimensionality of $\mathfrak{E}$, a countable tree $ T $ and a family of sets $\mathcal{E}=\{ E_s :s\in T\}$ 
which are closed with respect to $\mathcal{W}$ which satisfy the conditions of Theorem \ref{EE} for $\mathfrak{E}$.
Let $\mathcal{W}^n$ the topology of $(\mathfrak{E},\mathcal{W})^n$, $T^n=\{s_1*\ldots *s_n: s_1, \ldots, s_n\in T \textit{ and } \vert s_1\vert =\ldots =\vert s_n\vert\}$ and for each $s_1*\ldots *s_n\in T^n$ let $E_{s_1*\ldots* s_n}$ be the subset $E_{s_1}\times \ldots \times E_{s_n}$ of $\mathfrak{E}^n$. 

\begin{lema}\label{at}
The collections $\mathcal{W}^n$, $T^n$ and $\mathcal{E}^n=\{E_{s_1*\ldots* s_n}: s_1*\ldots *s_n\in T^n \}$ satisfy the
conditions of Theorem \ref{EE} for $\mathfrak{E}$.
\end{lema}

\begin{proof}
Since $T$ is a tree over the countable alphabet $A$, $T^n$ is a tree over the countable alphabet $A^n$ (Definition \ref{tree}). It is not difficult to prove that $\mathcal{W}^n$ witnesses that $\mathfrak{E}^n$ is almost zero-dimensional.
Moreover, it is clear that each $E_{s_1*\ldots* s_n}$ is closed in $\mathfrak{E}^n$. On the other hand $E_{\emptyset}=\mathfrak{E}$, so $E_{\emptyset *\ldots* \emptyset}=E_{\emptyset} \times \ldots \times E_{\emptyset}=\mathfrak{E}^n$. Furthermore,
$$E_{s_1}\times \ldots \times  E_{s_n}= \bigcup \{E_{t_1}: t_1\in  succ(s_1)\}\times \ldots \times \bigcup\{E_{t_n}: t_1\in  succ(s_n)\}= $$
$$\bigcup\{E_{t_1}\times \ldots \times  E_{t_n}: t_1*\ldots* t_n \in succ(s_1*\ldots * s_n)\}=$$
$$\bigcup\{E_{t_1*\ldots *t_n}: t_1*\ldots* t_n \in succ(s_1*\ldots * s_n)\}$$
(see Remark \ref{tree1}).
\\
Now we are going to prove that $\mathcal{W}^n$, $T^n$ and $\mathfrak{E}^n$ satisfy condition (2) of Theorem \ref{EE}. Let $(x_1,\ldots x_n)\in \mathfrak{E}^n$ and, for each $i\in \{1,\ldots n\}$, let $U_i$ be a neighborhood of $x_i$ which is an anchor for $\mathcal{E}$. Let $\widehat{t}\in T^n$; then $\widehat{t}=(\widehat{t_1},\ldots,\widehat{t_n})$ with $\widehat{t_i}\in [T]$ for each $i\in\{1,\ldots, n\}$ (Remark \ref{tree1}). We define $J=\{i\in \{1,\ldots, n\}: \textit{ there exists } m\in \N \textit{ such that } E_{\widehat{t_i}\downharpoonleft m } \cap U_i=\emptyset\}$. First assume that $J\neq \emptyset$; let $k\in J$, then $E_{(\widehat{t_1},\ldots ,\widehat{t_n})\downharpoonleft k} \cap(U_1\times \ldots\times U_n)=(E_{\widehat{t_1}\downharpoonleft k}\times \ldots \times E_{\widehat{t_n}\downharpoonleft k}\cap(U_1\times \ldots\times U_n)= ( E_{\widehat{t_1}\downharpoonleft k} \cap U_1)\times \ldots \times( E_{\widehat{t_n}\downharpoonleft k} \cap U_n)=\emptyset.$ Now assume that $J=\emptyset$. Since $U_i$ is an anchor for $\mathcal{E}$ for each $i\in \{1,\ldots n\}$, the
sequence $(E_{\widehat{t_i}\downharpoonleft j} )_{j< \omega}$ converges to a point $z_i$ in $\mathfrak{E}$. Therefore, the sequence  $(E_{\widehat{t_1}\downharpoonleft j} \times \ldots \times E_{\widehat{t_n}\downharpoonleft j} )_{j< \omega}$ converges to $(z_1,\ldots, z_n)\in \mathfrak{E}^n$. But $E_{\widehat{t_1}\downharpoonleft j} \times \ldots \times E_{\widehat{t_n}\downharpoonleft j} = E_{\widehat{t_1}\downharpoonleft j *\ldots * \widehat{t_n}\downharpoonleft j}=E_{(\widehat{t_1}\downharpoonleft j ,\ldots , \widehat{t_n}\downharpoonleft j)} = E_{\widehat{t}\downharpoonleft j}$ (see Remark \ref{tree1}). Hence, the sequence $(E_{\widehat{t}\downharpoonleft j})_{j<\omega}$ converges to $(z_1,\ldots, z_n)$.
We now verify condition (3) of Theorem \ref{EE}. Suppose $t_1*\ldots *t_n \in succ(s_1*\ldots * s_n)$, then for each $i\in \{1,\ldots n\}$
$t_i \in succ(s_i)$ (Remark \ref{tree1}). Thus, $E_{t_i}$ is nowhere dense in $E_{s_i}$ for each $i\in \{1,\ldots n\}$. This implies
that, $E_{t_1}\times \ldots \times E_{t_n}$ is nowhere dense in $E_{s_1}\times \ldots \times E_{s_n}$.
The condition (4) of Theorem \ref{EE} follows from the Remark 5.2 in \cite{jj}.
\end{proof}

\textbf{
Proof of Theorem 1.2}
\begin{proof} 
Let $ n \in \N$ be fixed. We are going to prove that $\mathcal{F}_n(\mathfrak{E})$ is homeomorphic to $\mathfrak{E}$ using Theorem $\ref{EE}$. Because of Lemma \ref{at} we know that if the topology $\mathcal{W}$, the tree $T$ and the family $\{E_s:s\in T\}$ satisfy the conditions of Theorem \ref{EE} for $\mathfrak{E}$, then the product topology $(\mathfrak{E}, \mathcal{W})^n$, which we denote here as $\mathcal{W}^n$, $T^n=\{s_1*\ldots*s_n:s_1,\ldots, s_n\in T\textit{ and } \vert s_1\vert =\ldots= \vert s_n\vert\}$ (Definition \ref{tree}) and the family $\{E_{s_1*\ldots*s_n}:s_1*\ldots*s_n \in T^n\}$ where $E_{s_1*\ldots*s_n}= E_{s_1}\times \ldots \times E_{s_n}$ for each $s_1,\ldots, s_n\in T$ 
with $\vert s_1\vert =\ldots= \vert s_n\vert$, satisfy all the conditions of Theorem \ref{EE}.

Let, $g:\mathfrak{E}^n\to \mathcal{F}_n(\mathfrak{E})$, be define by as $g(x_1,\ldots, x_n)=\{x_1, \ldots, x_n\}$. Let $\mathcal{W}^\prime$ be the Vietoris topology in $\mathcal{F}_n(\mathfrak{E}, \mathcal{W})$. The tree $T^\prime$ that we are going to consider is $T^\prime = T^n$, and the family $\mathcal{S}^\prime$ of subsets of $\mathcal{F}_n(\mathfrak{E})$ indexing by $T^n$ that we are going to prove to be closed with respect to $\mathcal{W}^n$ is $\mathcal{S}^\prime=\{H_{s_1*\ldots*s_n}: s_1*\ldots*s_n\in T^n \}$ where $H_{s_1*\ldots*s_n}:=g [E_{s_1*\ldots*s_n}]$ for each $s_1*\ldots*s_n\in T^n$. We will prove then that $\mathcal{W}^\prime$, $T^\prime$ and $\mathcal{S}^\prime$ satisfy the conditions required in Theorem  \ref{EE} for $\mathcal{F}_n(\mathfrak{E})$.

Indeed, the fact that the Vietoris topology in $\mathcal{F}(\mathfrak{E}, \mathcal{W})$ witnesses that $\mathcal{F}_n(\mathfrak{E})$ is almost zero-dimensional
fallows from the proof of Proposition \ref{HZD}. By Lemma \ref{A2}, $\mathcal{F}_n(\mathfrak{E})$ is $\{g[E_{s_1}\times \ldots\times E_{s_n}]:s_1*\ldots* s_n\in T^n\}$-cohesive. That is, $\mathcal{F}_n(\mathfrak{E})$ is $\{H_{s_1*\ldots*s_n}]:s_1*\ldots* s_n\in T^n\}$-cohesive. Moreover, $T^\prime = T^n$ is a tree over a countable alphabet. On the other hand, for each $s \in T$, $E_s$ is a closed
subset of $(\mathfrak{E}, \mathcal{W})$, hence for $s_1, \ldots , s_n \in T$ satisfying $\vert s_1\vert= \ldots= \vert s_n \vert$, $E_{s_1}\times \ldots\times E_{s_n}$ is closed in $(\mathfrak{E}, \mathcal{W})^n$. Additionaly, since the function $g$ is closed, $g[E_{s_1}\times \ldots\times E_{s_n}]$ is closed in $(\mathcal{F}_n(\mathfrak{E}),\mathcal{W}^\prime)$.

Now we are going to prove that $\mathcal{W}^\prime$, $T^\prime$ and $\mathcal{S}^\prime$ satisfy conditions (1), (2) and (3) of Theorem \ref{EE}. 
For $\emptyset=\emptyset *\ldots * \emptyset \in T^n$, $H_{\emptyset *\ldots * \emptyset }=g[E_{\emptyset}\times \ldots \times E_{\emptyset}]=g[\mathfrak{E}]=\mathcal{F}_n(\mathfrak{E})$. Let $s_1*\ldots * s_n$ be an
element in $T^n$. 
Using Lemma \ref{at}, we have that $H_{s_1*\ldots *s_n}:=g [E_{s_1*\ldots *s_n}]=g[\bigcup \{E_{t_1*\ldots *t_n}: t_1*\ldots *t_n \in succ(s_1*\ldots *s_n)\}]= \bigcup \{g[E_{t_1*\ldots *t_n}]: t_1*\ldots *t_n \in succ(s_1*\ldots *s_n)\}=\bigcup \{H_{t_1*\ldots *t_n}: t_1*\ldots *t_n \in succ(s_1*\ldots *s_n)\} $. This proves that $\mathcal{W}^\prime$, $T^\prime$ and $\mathcal{S}^\prime$ satisfy condition (1) of Theorem \ref{EE}.

In order to prove that $\mathcal{W}^\prime$, $T^\prime$ and $\mathcal{S}^\prime$ satisfy condition (2) of Theorem \ref{EE}, we take $F\in \mathcal{F}_n(\mathfrak{E})$.
Assume that $F=\{x_1,\ldots, x_n\}$. For each $j\leq n$ there exists a neighborhood $U_j$ of $x_j$ which is an
anchor for $\mathfrak{E}$. Let $ t \in  [T^\prime]$; then $ t = (t_1,\ldots, t_n)$ with $t_i\in [T]$ for each $i \in \{1,\ldots n\}$ (Remark \ref{tree1}).
We define $J = \{i\in\{1,\ldots n\}: \textit{ there exists } m \in \N  \textit { such that } E_{t_i\downharpoonleft m}\cap U_i = \emptyset\}$.
Assume that $J\neq\emptyset$; let $k \in J$. Then $U_i \cap E_{t_i\downharpoonleft k }= $; for all $i, j\in \{1,\ldots, n\}$. Hence, $g^\leftarrow(F) \
(E_{t_1\downharpoonleft k }\times \ldots\times E_{t_n\downharpoonleft k }) = \emptyset$. Thus, there exists neighborhoods $ V_j $ of $x_j$ for each $j $ such that if
$w \in g(F)$, then there exists a permutation $h:\{1,\ldots n\}\to \{1,\ldots n\}$, such that $w \in V_h =
V_{h(1)}\times \ldots \times V_{h(n)} \subset \mathfrak{E}^n\setminus ( E_{t_1\downharpoonleft k }\times \ldots \times E_{t_n\downharpoonleft k })$.
Therefore, $V =\bigcup_{h\in P} V_h  \subset \mathfrak{E}^n\setminus ( E_{t_1\downharpoonleft k }\times \ldots \times E_{t_n\downharpoonleft k })$ (where $P$ is the set of permutations of $\{1,\ldots, n\}$). Let $\mathcal{U}= \langle W_1,\ldots, W_n\rangle$, where $W_i = U_i\cap V_i$, thus $F\in \mathcal{U}$, and $g(E_{t_1\downharpoonleft k }\times \ldots \times E_{t_n\downharpoonleft k })\cap \mathcal{U} =\emptyset $.
That is $H_{t\downharpoonleft k}\cap \mathcal{U} = \emptyset$. Then $ \mathcal{U}$ is an anchor.
Now suppose that $J = \emptyset$. For every $i \in \{1,\ldots n\}$, since $U_i$ is an anchor, the sequence $(E_{t_i\downharpoonleft j} )_{j<\omega}$
converges to a point $y_i$. Thus, $(E_{t_1\downharpoonleft j}\times \ldots \times E_{t_n\downharpoonleft j} )_{j<\omega}$ converges to $(y_1,\ldots, yn)$. Since $g$ is
continuous by the Lemma \ref{ancla}, we have that $(g((E_{t_1\downharpoonleft j}\times \ldots \times E_{t_n\downharpoonleft j}) )_{j<\omega}$ converges to $g(y_1,\ldots, y_n)$. That is, $(H_{t\downharpoonleft j})_{j<\omega}$
converges to $g(y_1,\ldots,  y_n)$.
Finally, we will prove that$\mathcal{W}^\prime$, $T^\prime$ and $\mathcal{S}^\prime$ satisfy condition (3) of Theorem \ref{EE}. If $t_1* \ldots * tn \in succ(s_1*\ldots * s_n)$, then for each $i \in \{1,\ldots n\}$, $t_i \in succ(s_i)$ (Remark \ref{tree1}). Thus, $E_{t_i}$ is nowhere dense in $E_{s_i}$ for each $i \in \{1,\ldots n\}$. This implies that, for each permutation $h: \{1,\ldots n\}\to  \{1,\ldots n\}$,
$E_{t_{h(1)}} \times \ldots \times E_{t_{h(n)}}$ is nowhere dense in $E_{s_{h(1)}} \times \ldots \times E_{s_{h(n)}}$. Therefore,
$\bigcup_{h\in P} (E_{t_{h(1)}} \times \ldots \times E_{t_{h(n)}})$ is nowhere dense in $\bigcup_{h\in P} (E_{s_{h(1)}} \times \ldots \times E_{s_{h(n)}})$ 
 where $P$ is the set of permutations of $\{1,\ldots n\}$.
Note that 
$$g^\leftarrow[g[E_{t_1}\times \ldots \times E_{t_n}]]= \bigcup_{h\in P} (E_{t_{h(1)}} \times \ldots \times E_{t_{h(n)}})\subset g^\leftarrow[g[E_{s_1}\times \ldots \times E_{s_n}]]$$
and
$$g^\leftarrow[g[E_{s_1}\times \ldots \times E_{s_n}]]= \bigcup_{h\in P} (E_{s_{h(1)}} \times \ldots \times E_{s_{h(n)}}).$$
 As the finite union of nowhere dense sets is nowhere
dense, by the Lemma \ref{dnp} $g[E_{t_1}\times \ldots \times E_{t_n}]$ is nowhere dense in $g[E_{s_1}\times \ldots \times E_{s_n}]$. By Theorem \ref{EE}, all the above proves that $\mathcal{F}_n(\mathfrak{E})$ is homeomorphic to $\mathfrak{E}$.
\end{proof}

\section{ Applications}

\begin{defi}
Let $X$ be a space. 
 A space $ Y $ is called an $ X $ -factor, if there is a space $ Z $ such that
$ Y \times Z $ is homeomorphic to $ X $.
\end{defi}

It is known that $X$ is a $\mathfrak{E}_{\mathrm{c}}$-factor or a $\mathfrak{E}$-factor if only if $X$ can be embedded as a closed
subset of $\mathfrak{E}_{\mathrm{c}}$ or $\mathfrak{E}$ respectively (see {\cite[Theorem 3.2]{j}} and {\cite[Theorem 9.2]{jj}}). Using these facts we have the following corollary.

\begin{cor}
Let $n\in \N$, if $X$ is a $\mathfrak{E}_{\mathrm{c}}$-factor or a $\mathfrak{E}$-factor, then $\mathcal{F}_n(X)$ is a $\mathfrak{E}_{\mathrm{c}}$-factor or $\mathfrak{E}$-factor 
respectively.
\end{cor}
\begin{proof}
If $X$ is a $\mathfrak{E}_{\mathrm{c}}$-factor or a $\mathfrak{E}$-factor, then $X$ can be embedded as a closed subset of $\mathfrak{E}_{\mathrm{c}}$ or $\mathfrak{E}$ respectively. Then $\mathcal{F}_n(X)$ can be embedded as a closed subset of $\mathcal{F}_n(\mathfrak{E}_{\mathrm{c}})$ or $\mathcal{F}_n(\mathfrak{E})$ respectively. By Theorem \ref{EE1} and Theorem \ref{ee1}, we have that $\mathcal{F}_n(\mathfrak{E}_{\mathrm{c}})$ and $\mathcal{F}_n(\mathfrak{E})$ 
are homeomorphic to $\mathfrak{E}_{\mathrm{c}}$ and $\mathfrak{E}$ respectively. Thus $\mathcal{F}_n(X)$ is $\mathfrak{E}_{\mathrm{c}}$-factor or $\mathfrak{E}$-factor respectively.
\end{proof}

Another known fact is that any open subset of $\mathfrak{E}$ or $\mathfrak{E}_{\mathrm{c}}$ 
is homeomorphic to $\mathfrak{E}$ or $\mathfrak{E}_{\mathrm{c}}$ respectively (see {\cite[Theorem 3.2]{j}} and {\cite[Corollary 8.15]{jj}}). Using that fact, and Theorems \ref{EE1} and \ref{ee1} we have the following corollary.

\begin{cor}
Let $X_1=\mathfrak{E}$ and $X_2=\mathfrak{E}_{\mathrm{c}}$ and $n\in \N$
For each $k< n$ we have that $\mathcal{F}_{n}(X_i)\setminus \mathcal{F}_{k}(X_i) $ 
is homeomorphic to $X_i$.
\end{cor}
\begin{proof}
This follows from Theorems \ref{EE1} and \ref{ee1} and the fact that $\mathcal{F}_{n}(X_i)\setminus \mathcal{F}_{k}(X_i) $ is an open subset of $\mathcal{F}_{n}(X_i)$.
\end{proof}

To finish we will mention some comments about $\mathcal{F}(\mathfrak{E})$, $\mathcal{F}(\mathfrak{E}_{\mathrm{c}})$, $\mathcal{K}(\mathfrak{E})$ and $\mathcal{K}(\mathfrak{E}_{\mathrm{c}})$. 
As $\mathcal{F}(\mathfrak{E}_{\mathrm{c}})$ 
is of the first category and $\mathfrak{E}_{\mathrm{c}}$ is not, then $\mathcal{F}(\mathfrak{E}_{\mathrm{c}})$ cannot be homeomorphic to $\mathfrak{E}_{\mathrm{c}}$. 
Neither does it happen that $\mathcal{F}_n(\mathfrak{E}_{\mathrm{c}})$ is not homeomorphic to $\mathfrak{E}$ because
$\mathcal{F}_n(\mathfrak{E}_{\mathrm{c}})$ is $G_{\delta\sigma}$ and $\mathfrak{E}$ is not. The space $\mathfrak{E}$ is not $G_{\delta\sigma}$ because it contains a closed copy of $\mathfrak{E}$.

On the other hand by Theorems \ref{EE1} and \ref{ee1}, if $X_1=\mathfrak{E}$ or $X_2=\mathfrak{E}_{\mathrm{c}}$, then $X_i^k$ 
is homeomorphic to $\mathcal{F}_{k}(X_i)$ for each $k\in \N$. 
Also we have $\mathcal{F}(X_i)=\bigcup_{k\in \N}\mathcal{F}_{k}(X_i)$ and let $a_i$ be a fixed point in $X_i$ for $i \in \{1,2\}$. For each $k \in\N$ and each $i \in\{1,2\}$, let $Z_k^i$ be the subspace $\{(x_j)_{j\in \omega} \in X_i^\omega: x_j=a_i,  j>k\}$ of $X^\omega_i$. Then $\mathcal{F}_{k}(X_i)$ is homeomorphic to $Z_k^i$ from this the following question arises.
\begin{pre}
Let $X_1=\mathfrak{E}$ and $X_2=\mathfrak{E}_{\mathrm{c}}$. is  $\mathcal{F}(X_i)$ homeomorphic to the subspace
$\bigcup_{k\in \N} Z^k_i$ of $X^\omega_i$ for $i = 1, 2$?
\end{pre}

Also, we note that $\mathcal{K}(\mathfrak{E})$ is homeomorphic to neither $\mathfrak{E} $ nor $\mathfrak{E}_{\mathrm{c}}$, and moreover $\mathcal{K}(\mathfrak{E})$ is not a factor of either spaces.
Because $\mathfrak{E}$ has a closed copy of $\Q$ (see \cite{KKK}), then $\mathcal{K}(\mathfrak{E})$ has a closed copy of $\mathcal{K}(\Q)$. On the other hand it is known that $\mathcal{K}(\Q)$ is not a Borel set (see \cite{KKK}), thus $\mathcal{K}(\mathfrak{E})$ 
is not a Borel set. But $\mathcal{K}(\mathfrak{E}_c)$ is factor of $\mathfrak{E}$ since it is possible to embedd it as a closed subset of $\mathfrak{E}$. So far we only know that space $\mathcal{K}(\mathfrak{E}_c)$ is Polish, $AZD$ and cohesive, what we have to prove is whether $\mathcal{K}(\mathfrak{E}_c)$ satisfy property  $C_1$. So have the fallowing question:

\begin{pre}

Is $\mathcal{K}(\mathfrak{E}_c)$ homeomorphic to $\mathfrak{E}_{\mathrm{c}}$ or $\mathfrak{E}_{\mathrm{c}}^\omega$?
\end{pre}

\end{document}